\documentclass[a4paper,12pt]{article}
\usepackage[centertags]{amsmath}
\usepackage{amsfonts}
\usepackage{amssymb}
\usepackage{amsthm}
\usepackage{dsfont}
\usepackage{tikz, subfigure}
\usepackage{verbatim}

\newtheorem{Theorem}{Theorem}[section]
\newtheorem{Proposition}{Proposition}[section]

\newtheorem{Remark}{Remark}[section]

\newtheorem{Claim}{Claim}[section]
\newtheorem{Question}{Question}[section]
\AtEndDocument{\bigskip{\footnotesize%
  \textsc{Department of Mathematics, Ningbo University, Ningbo, Zhejiang, People's Republic of China} \par
  \textit{E-mail address:} \texttt{kanjiangbunnik@yahoo.com, jiangkan@nbu.edu.cn} \par
}}
\AtEndDocument{\bigskip{\footnotesize%
  \textsc{Department of Mathematics, Ningbo University, Ningbo, Zhejiang, People's Republic of China} \par
  \textit{E-mail address:} \texttt{akevbgf@hotmail.com} \par
}}
\AtEndDocument{\bigskip{\footnotesize%
  \textsc{Department of Mathematics, Ningbo University, Ningbo, Zhejiang, People's Republic of China} \par
  \textit{E-mail address:} \texttt{xilifengningbo@yahoo.com} \par
}}

\begin{document}
\title{xxxx}
\date{}
\title{On arithmetic progressions in self-similar sets}
\author{Kan Jiang, Qiyang Pei and Lifeng Xi \thanks{Corresponding author}}
\maketitle{}
\begin{abstract}
Given a sequence $\{b_{i}\}_{i=1}^{n}$ and a ratio $\lambda
\in (0,1),$ let $E=\cup_{i=1}^n(\lambda E+b_i)$  be a homogeneous self-similar set.
In this paper, we study the existence and maximal length of arithmetic
progressions in $E$. Our main idea is from the multiple $\beta$-expansions.
 
\noindent   Key words: Arithmetic progressions;  self-similar sets; $\beta$-expansions 

\noindent  AMS Subject Classifications:  28A80, 28A78.
\end{abstract}

\section{Introduction}
An arithmetic progression in $\mathbb{R}$ is of the form
\begin{equation*}
P=\{a,a+\delta ,a+2\delta ,a+3\delta ,\cdots ,a+(k-1)\delta \}
\end{equation*}%
for some $a\in \mathbb{R}$, $\delta \in \mathbb{R}^{+}$ and $k\in \mathbb{N}%
^{+}$. We say $P$ is an arithmetic progression with length $k$. Finding an
arithmetic progression in a subset of $\mathbb{R}$ is a hot spot in
combinatorial number theory and ergodic theory. Under some conditions, the
existence of an arithmetic progression of a set inspires many scholars to
investigate. In the setting of discrete case, Erd\H{o}s and Tur\'{a}n \cite{ErdosPaul}
conjectured a subset of natural numbers with positive density necessarily
contains arbitrarily long arithmetic progressions. This conjecture was
solved by Roth  \cite{Roth53} if the length of the arithmetic progression is $3$.  Roth and other scholars tried to prove  the  Erd\H{o}s-Tur\'{a}n conjecture, but without full success. 
16 years later, 
Szemer\'{e}di \cite{Szemeredi1},  used  the purely combinatorial methods, proved that the Erd\H{o}s-Tur\'{a}n conjecture holds if the length of the  arithmetic progression is $4.$ In \cite{Szemeredi2}, Szemer\'{e}di  extended Roth's theorem to arbitrarily long arithmetic progressions, and completely  addressed the Erd\H{o}s-Tur\'{a}n conjecture.  From then on, many different proofs of Szemer\'{e}di's theorem were found. For instance, Furstenberg \cite{Furstenberg} proved that the Szemer\'{e}di's theorem is equivalent to the multiple recurrence theorem in ergodic theory, and  gave a proof of the multiple recurrence theorem. Therefore, he obtained a new proof of Szemer\'{e}di's theorem, for a survey of this topic, see \cite{TTT}.
 It is natural to investigate to  a subset of natural
numbers  which is of  zero density. In this case,  such set may still contain arbitrarily long
arithmetic progressions. For instance, Green and Tao \cite{GT1} proved their celebrated
Green-Tao theorem for the primes.

In the setting of continuous case, Steinhaus proved that for any set $
E\subset \mathbb{R}$ with positive Lebesgue measure must contain arbitrarily
long arithmetic progressions. Steinhaus' result is a consequence of the Lebesgue density theorem. {\L}aba and Pramanik \cite{Laba2016}, using some techniques from  Fourier analysis,  proved that if a closed
subset $E$ of $\mathbb{R}$ with Hausdorff dimension that is close to one,
and $E$ supports a probability measure which obeies appropriate Fourier
decay and mass decay, then $E$ contains non-trivial arithmetic progressions
with length $3$.  Shmerkin \cite{Shmerkin2017}  constructed a   Salem set in $\mathbb{R}\setminus \{Z\}$ which does not include any arithmetic progression with length $3$. 
 Fraser and Yu \cite{Yu2016} proved that for any one-dimensional set, if its Assouad dimension is strictly smaller than 1, then it cannot  contain arbitrarily
long arithmetic progressions. 
Li,   Wu and Xiong \cite{LWX} proved that for a special class of  Moran set, it  contains arbitrarily
long arithmetic progressions if and only if the Assouad dimension of the associated set is $1.$ Recently, 
Chaika \cite{Jon} proved that for a class of middle-$N$th Cantor set, when the contracitve ratio tends to the $1/2$,  the length of arithmetic progressions goes to infinity. Chaika's main idea, to prove the existence of the arithmetic progressions, is using the  intersection of the Cantor set with its  multiple  translations. 
 Later, Broderick, Fishman and Simmons \cite{BFS} proved the quantitative result of the length of the arithmetic progressions. They gave the approximated length of the arithmetic progressions when the contracitve ratio tends to the $1/2$.
Their idea was motivated by the Schmidt's game, which is  a very useful tool in the setting of Diophantine approximation. 

In this paper, we shall consider the arithmetic progressions in self-similar sets.  We first review the main result of  Broderick, Fishman and Simmons. 
Let $K_{\epsilon}$ be the attractor  of the IFS $$\left\{f_1(x)=\dfrac{1-\epsilon}{2}x, f_2(x)=\dfrac{1-\epsilon}{2} x+\dfrac{1+\epsilon}{2}\right\}, 0<\epsilon<1.$$
Broderick, Fishman and Simmons \cite{BFS}  proved the following result. 
\begin{Theorem}\label{Motivation}
Let $L_{AP}(K_{\epsilon})$ denote the maximal length of an arithmetic progression in  $K_{\epsilon}$. Then for all $\epsilon>0$ sufficiently small and $n\in \mathbb{N}$ sufficiently large, we have 
$$\dfrac{1/\epsilon}{\log 1/\epsilon}\lesssim L_{AP}(K_{\epsilon})\leq 1/\epsilon +1$$
where $A\lesssim B$ means that there exists a constant  $C$ such that $A\leq CB.$ Moreover, 
$L_{AP}(K_{\epsilon})\geq 4$ if and only if $0<\epsilon\leq 1/3$. 
\end{Theorem}
\noindent In this paper, we shall generalize the second result of Theorem \ref{Motivation}. 

Before we state the main theorem, we introduce some definitions and results.
For $\mathbf{b}=\{b_{i}\}_{i=1}^{n}$ with $b_{1}<b_{2}<\cdots <b_{n},$ let
\begin{equation*}
E_{\lambda ,\mathbf{b}}=\bigcup\nolimits_{i=1}^{n}(\lambda E_{\lambda ,
\mathbf{b}}+b_{i})
\end{equation*}
denote the self-similar set with respect to the IFS $\{f_{i}(x)=\lambda
x+b_{i}\}_{i=1}^{n}, $ for the definition of self-similar sets, see \cite{FG,Hutchinson}. For a general self-similar set, its IFS may have overlaps, i.e. the IFS does not  satisfies the open set condition or strong separation condition (the definitions can be found in \cite{FG,Hutchinson}). However, in this case, 
it is easy to obtain the following result (the proof is in the next section).
\setcounter{Proposition}{1}
\begin{Proposition}\label{ssc}
If the strong separation condition fails for $E_{\lambda ,\mathbf{b}},$ then
there are three-term A.P.(arithmetic progression) contained in $E_{\lambda ,\mathbf{b}}$.
\end{Proposition}
 Due to this result, it is natural to consider  the arithmetic progressions in   self-similar sets with the strong separation condition. 
 
Note that $E_{\lambda ,\mathbf{b}}=\frac{b_{n}-b_{1}}{1-\lambda }E_{\lambda ,
\mathbf{\bar{b}}}+b_{1}\left( \sum\nolimits_{n=0}^{\infty }\lambda
^{n}\right) ,$ where $\mathbf{\bar{b}}=(\bar{b}_{1},\cdots ,\bar{b}_{n})$
satisfying $0=\bar{b}_{1}<\bar{b}_{2}<\cdots <\bar{b}_{n}=1-\lambda .$ For a
subset $F$ on the line, let $$L_{AP}(F)=\sup \{k:\exists \{c_{i}\}_{i=1}^{k}  \mbox{ is
an A.P. contained in } F\}.$$ Then  $L_{AP}(E_{\lambda ,\mathbf{b}})=L_{AP}(E_{\lambda ,
\mathbf{\bar{b}}}).$ Without loss of generality we may assume that
\begin{equation}
0=b_{1}<b_{2}<\cdots <b_{n}=1-\lambda .  \label{haha}
\end{equation}
In terms of the following result (the proof is available in the next section), we may assume that  $\mathbf{b}=\{b_{i}\}_{i=1}^{n}$ is an arithmetic progression. 
\begin{Proposition}\label{10}
Suppose there is no any A.P. in $\{b_{i}\}_{i=1}^{n}$ satisfying (\ref{haha})%
$.$ If $$\lambda <\left( \min_{i<j<k}|2b_{j}-b_{i}-b_{k}|\right) /2,$$ then
there is no A.P. in $E_{\lambda ,\mathbf{b}}$.
\end{Proposition}
Given an A.P.  $\mathbf{b}=\{b_{i}\}_{i=1}^{n}$ and a ratio $\lambda \in
(0,1/n),$ we obtain a self-similar set $E_{\lambda ,\mathbf{b}}$ satisfying
the strong separation condition. As the discussion above, we assume that
\begin{equation}
b_{i}=(i-1)(\lambda +\alpha )\text{ with }\alpha =\frac{1-n\lambda }{n-1}.
\label{a}
\end{equation}
Now let $E_{\lambda }^{n}=E_{\lambda ,\mathbf{b}}$ where $\{b_{i}\}_{i=1}^{n}
$ is defined in (\ref{a}). In particular, for $n=2$, the self-similar set $%
E_{\lambda }^{n}$ is the middle-$\alpha $ Cantor set with $\alpha
=1-2\lambda .$ In this paper, we shall investigate the arithmetic progressions in the attractor $E_{\lambda }^{n}.$

Note that there is a natural A.P. $\{b_{i}\}_{i=1}^{n}$ in $E_{\lambda }^{n},
$ i.e., $L_{AP}(E_{\lambda }^{n})\geq n.$ An important problem is
\begin{center}
how about the estimate of $L_{AP}(E_{\lambda }^{n})?$
\end{center}
Now we state the main result of this paper.  
\setcounter{Theorem}{3}
\begin{Theorem}\label{Main}
\label{Main copy(1)} Consider the self-similar set $E_{\lambda }^{n}$, the
followings are equivalent,

$(1)$ $L_{AP}(E_{\lambda }^{n})\geq n+1;$

$(2)$ $\lambda \geq \frac{1}{2n-1};$

$(3)$ $L_{AP}(E_{\lambda }^{n})\geq 2n.$ \newline
In particular, $L_{AP}(E_{\lambda }^{n})=n$ if and only if $\lambda \in (0,\frac{1}{%
2n-1}).$
\end{Theorem}
\setcounter{Remark}{4}
\begin{Remark}
In \cite{BFS}, Shmerkin pointed that $L_{AP}(K_{\epsilon})\geq 4$ if and only if $0<\epsilon\leq 1/3$ can be proved by the gap lemma. We, however, will use some  basic ideas from $\beta$-expansions \cite{GS,MK,DJKL,KKLL2017,DajaniDeVrie} to find the arithmetic progressions. Moreover, our proof is  constructive. We note  that  finding the arithmetic progressions in $E_{\lambda }^{n}$ is essentially a problem in the setting of $\beta$-expansions, i.e. given a point in some interval, then how can we find its expansions in base $1/\lambda.$ Nevertheless, for the multiple  $\beta$-expansions, to the best of our knowledge, there are few results \cite{DJKL,KarmaKan2}.  This is the main reason which makes  the constructive proof  difficult. 
\end{Remark}
This paper is arranged as follows. In section 2, we give a proof of Theorem \ref{Main}. Moreover, we also prove other useful results which estimate the upper and lower bounds of $L_{AP}(E_{\lambda }^{n}).$ In section 3, we pose one problem.
\section{\textbf{Proof of Theorem \ref{Main}}}
Before we prove Theorem \ref{Main}, we prove some results concerning with the lower and upper bound of $l(E_{\lambda }^{n}).$
\begin{proof}[\textbf{Proof of Proposition \ref{ssc}}]
Suppose $f_{i}(E_{\lambda ,\mathbf{b}})\cap f_{i+1}(E_{\lambda ,\mathbf{b}})\neq
\emptyset $ with $b_{i}<b_{i+1},$ we take a point $y=f_{i}(u)=f_{i+1}(v)$
with $u,v\in E_{\lambda ,\mathbf{b}}$ in this intersection and let $%
x=f_{i}(v)$ and $y=f_{i+1}(u).$ Then $x-y=y-z,$ that means $%
\{x,y,z\}(\subset E_{\lambda ,\mathbf{b}})$ is an A.P.
\end{proof}
\begin{proof}[\textbf{Proof of Proposition \ref{10}}]
Suppose on the contrary that there exists an A.P. $x<y<z$ contained in $%
E_{\lambda ,\mathbf{b}}.$ Let $x\in f_{i}(E_{\lambda ,\mathbf{b}}),$ $y\in
f_{j}(E_{\lambda ,\mathbf{b}})$ and $z\in f_{k}(E_{\lambda ,\mathbf{b}}),$
then $x\in \lbrack b_{i},b_{i}+\lambda ],$ $y\in \lbrack b_{j},b_{j}+\lambda
]$ and $z\in \lbrack b_{k},b_{k}+\lambda ]$ which implies $|2y-x-z|\geq
|2b_{j}-b_{i}-b_{k}|-2\lambda >0.$ It is a contradiction.
\end{proof}
For the upper bound of $L_{AP}(E_{\lambda }^{n}),$ using the
self-similarity, we have
\begin{Proposition}\label{3}
\label{common}For any $k\leq L_{AP}(E_{\lambda }^{n}),$ we have an A.P. of length
$k$ with common difference $d\geq \alpha .$ As a result,
\begin{equation*}
L_{AP}(E_{\lambda }^{n})\leq \left[ \frac{1}{\alpha }\right] +1=\left[ \frac{%
1-n\lambda }{n-1}\right] +1.
\end{equation*}
\end{Proposition}
\setcounter{Remark}{1}
\begin{Remark}
This proposition was also proved in \cite{BFS}.
\end{Remark}
We also have another estimate of upper bound of $L_{AP}(E_{\lambda }^{n}).$
\setcounter{Proposition}{2}
\begin{Proposition}\label{4}
Let $\lambda _{n,m}\in (0,1)$ be the solution of $1=nx+(n-1)x^{m}.$ Then $
L_{AP}(E_{\lambda }^{n})\leq n^{m}$ if $\lambda \leq \lambda _{n,m}.$ In
particular, for $n=2,$ we obtain that $L_{AP}(E_{\lambda }^{2})\leq 2^{m}$ if $
\lambda \leq \lambda _{2,m}.$ Subsequently, we have that if $1/3\leq \lambda<\sqrt{2}-1$, then $L_{AP}(E_{\lambda }^{2})=4.$
\end{Proposition}
\begin{proof}[\textbf{Proof of Proposition \ref{3}}]
Let $E_{\lambda }^{n}$ be the invariant set of the IFS $$\{S_{i}(x)=\lambda
x+(i-1)(\lambda +\alpha )\}_{i=1}^{n}.$$
Suppose that $\{c_{i}\}_{i=1}^{k}\subset E_{\lambda }^{n}$ is an A.P., and $%
I_{i_{1}\cdots i_{t}}=S_{i_{1}}\cdots S_{i_{t}}([0,1])$ is the smallest
basic interval containing $\{c_{i}\}_{i=1}^{k}.$ Hence there exists two
different $i_{t+1},i_{t+1}^{\prime }$ such that
\begin{equation*}
I_{i_{1}\cdots i_{t}i_{t+1}}\cap \{c_{i}\}_{i=1}^{k}\neq \emptyset \text{
and }I_{i_{1}\cdots i_{t}i_{t+1}^{\prime }}\cap \{c_{i}\}_{i=1}^{k}\neq
\emptyset .
\end{equation*}%
Then $\{S_{i_{1}\cdots i_{t}}^{-1}(c_{i})\}_{i=1}^{k}$ is also an A.P. in $%
E_{\lambda }^{n}$ with common difference equal or greater than
\begin{equation*}
\text{d}(I_{i_{t+1}},I_{i_{t+1}^{\prime }})\geq \alpha .
\end{equation*}
\end{proof}
\begin{proof}[\textbf{Proof of Proposition \ref{4}}]
Suppose on the contrary that $L_{AP}(E_{\lambda }^{n})> n^{m}.$ By pigeonhole
principle, there are two point of the A.P. lying in a basic interval of
length $\lambda ^{m}.$

On the other hand, in the same way as above, the common difference is equal
or greater than
\begin{equation*}
\alpha =\frac{1-n\lambda }{n-1}<\lambda ^{m}
\end{equation*}
if $\lambda \leq \lambda _{n,m}.$ Hence we obtain a contradiction.
Combining  with  Theorem \ref{Motivation} or Theorem \ref{Main}, we obtain the last statement. 
\end{proof}
 Now we give a proof of Theorem \ref{Main}.

\textbf{Step 1}. $(3)\Longrightarrow (1).$

It is obvious.

\medskip

\textbf{Step 2}. $(1)\Longrightarrow (2).$

Using Proposition \ref{common}, we have $d\geq \alpha =\frac{1-n\lambda }{n-1%
}.$ On the other hand, by pigeonhole principle, there are two points in the
A.P. contained in a same basic interval with length $\lambda ,$ then
\begin{equation*}
\lambda \geq d\geq \frac{1-n\lambda }{n-1}.
\end{equation*}%
Hence $\lambda \geq \frac{1}{2n-1}.$

\medskip

\textbf{Step 3}. $(2)\Longrightarrow (3).$

We will construct an A.P. $\{c_{i}\}_{i=1}^{2n}$ such that $
\{c_{i}\}_{i=1}^{2n}\subset E_{\lambda }^{n}.$ Let
\begin{equation*}
k=\left[ \frac{1}{\lambda }\right] =\frac{1}{\lambda }-\left( \frac{1}{
\lambda }\right) .
\end{equation*}
where $\left[ x\right] $ and $\left( x\right) $ are integral and decimal
part of $x\geq 0$. Denote
\begin{equation*}
\beta =\left\{
\begin{array}{cc}
\frac{k}{2} & \text{if }k\text{ is even,} \\
\frac{k+1}{2} & \text{if }k\text{ is odd.}%
\end{array}%
\right.
\end{equation*}
Note that the invariant  set of   the IFS $\{S_{i}x=\lambda x+i(\lambda +\alpha )\}_{i=0}^{n-1}$   can be represented as follows
$$ E_{\lambda}^{n}=\left\{(\lambda +\alpha )\sum\nolimits_{t=1}^{\infty }a_{t}\lambda ^{t}:a_t\in \{0,1,\cdots,n-1\}\right\}.$$
In other words, if  $x\in E_{\lambda}^{n}$, then   $$x=(\lambda +\alpha )\sum\nolimits_{t=1}^{\infty }a_{t}\lambda ^{t}$$ for some $(a_t)\in \{0,1,\cdots,n-1\}^{\mathbb{N}}$. We call $(a_t)$ a coding of $x$.  For simplicity, we denote
$$x=a_1a_2a_3\cdots.$$
Suppose $\{c_{i}\}_{i=1}^{2n}\subset E_{\lambda}^{n}$ have codings w.r.t. the IFS as follows%
\begin{equation*}
c_{2i-1}=(i-1)0c_{2i-1}^{(3)}c_{2i-1}^{(4)}c_{2i-1}^{(5)}\cdots \text{ and }%
c_{2i}=(i-1)\beta c_{2i}^{(3)}c_{2i}^{(4)}c_{2i}^{(5)}\cdots \text{ for }%
i=1,\cdots ,n.
\end{equation*}%
That means
\begin{equation*}
c_{i}=(\lambda +\alpha )\sum\nolimits_{t=1}^{\infty }c_{i}^{(t)}\lambda ^{t}%
\text{ with }c_{2i-1}^{(1)}=i-1,c_{2i-1}^{(2)}=0\text{ and }%
c_{2i}^{(1)}=i-1,c_{2i}^{(2)}=\beta ,
\end{equation*}%
where
\begin{equation*}
c_{i}^{(t)}\in \{0,1,\cdots ,(n-1)\}.
\end{equation*}
To insure that $\{c_{i}\}_{i=1}^{2n}$ is an A.P., we only need to show that
\begin{equation}
\left\{
\begin{array}{c}
2c_{2}=c_{1}+c_{3}, \\
2c_{3}=c_{2}+c_{4}, \\
\cdots \\
2c_{2n-1}=c_{2n}+c_{2n-2}.
\end{array}
\right.  \label{0}
\end{equation}
In fact, for the equation $2c_{2i}=c_{2i-1}+c_{2i+1}$ we obtain that
\begin{equation}
2\beta \lambda ^{2}-\lambda =\sum\nolimits_{t=3}^{\infty }\left(
c_{2i-1}^{(t)}+c_{2i+1}^{(t)}-2c_{2i}^{(t)}\right) \lambda ^{t}.  \label{1}
\end{equation}
In the same way, for $2c_{2i+1}=c_{2i}+c_{2i+2},$ we have%
\begin{equation*}
\lambda -2\beta \lambda ^{2}=\sum\nolimits_{t=3}^{\infty }\left(
c_{2i}^{(t)}+c_{2i+2}^{(t)}-2c_{2i+1}^{(t)}\right) \lambda ^{t},
\end{equation*}
which implies
\begin{equation}
2\beta \lambda ^{2}-\lambda =\sum\nolimits_{t=3}^{\infty }\left(
2c_{2i+1}^{(t)}-c_{2i}^{(t)}-c_{2i+2}^{(t)}\right) \lambda ^{t}.  \label{2}
\end{equation}
\begin{Claim}
There is a sequence $\{d_{t}\}_{t=3}^{\infty }$ of integers such that
\begin{equation*}
2\beta \lambda ^{2}-\lambda =\sum\nolimits_{t=3}^{\infty }(2d_{t})\lambda
^{t},
\end{equation*}
where $d_{t}\in \{-(n-1),\cdots ,-1,0,1,\cdots ,(n-1)\}.$
\end{Claim}
(1) We first verify that
\begin{equation*}
|2\beta \lambda ^{2}-\lambda |\leq 2(n-1)\frac{\lambda ^{3}}{1-\lambda }.
\end{equation*}
\textbf{Case 1.} When $k$ is even, we shall check that
\begin{equation*}
0\geq 2\beta \lambda -1\geq -2(n-1)\frac{\lambda ^{2}}{1-\lambda },
\end{equation*}
where $2\beta =\frac{1}{\lambda }-\left( \frac{1}{\lambda }\right) .$ In
fact, since $\left( \frac{1}{\lambda }\right) <1$ and $\lambda \geq \frac{1}{
2n-1},$ we have
\begin{eqnarray*}
2(n-1)\frac{\lambda ^{2}}{1-\lambda }+2\beta \lambda -1 &=&-\left( \frac{1}{%
\lambda }\right) \lambda +2(n-1)\frac{\lambda ^{2}}{1-\lambda } \\
&\geq &\frac{\lambda }{1-\lambda }\left( (2n-1)\lambda -1\right) \geq 0.
\end{eqnarray*}

\textbf{Case 2.} When $k$ is odd, we need to show that%
\begin{equation*}
0\leq 2\beta \lambda -1\leq 2(n-1)\frac{\lambda ^{2}}{1-\lambda },
\end{equation*}%
where $2\beta =\frac{1}{\lambda }-\left( \frac{1}{\lambda }\right) +1.$ In
fact, since $\left( \frac{1}{\lambda }\right) \geq 0$ and $\lambda \geq
\frac{1}{2n-1},$ we have%
\begin{eqnarray*}
&&2(n-1)\frac{\lambda ^{2}}{1-\lambda }-2\beta \lambda +1 \\
&=&2(n-1)\frac{\lambda ^{2}}{1-\lambda }-(\frac{1}{\lambda }-\left( \frac{1}{%
\lambda }\right) +1)\lambda +1 \\
&\geq &2(n-1)\frac{\lambda ^{2}}{1-\lambda }-\lambda =\frac{\lambda }{%
1-\lambda }\left( (2n-1)\lambda -1\right) \geq 0.
\end{eqnarray*}

(2) It suffices to verify
\begin{equation*}
J_{i}\cap J_{i+1}\neq \emptyset ,\text{ }
\end{equation*}
where $i\in \{-(n-1),\cdots ,-1,0,1,\cdots ,(n-2)\}$ and
\begin{equation*}
J_{i}=2\left[i\lambda +\frac{-(n-1)\lambda ^{2}}{1-\lambda },i\lambda +\frac{%
(n-1)\lambda ^{2}}{1-\lambda }\right].
\end{equation*}%
In fact, we only need to check that $\lambda +\frac{-(n-1)\lambda ^{2}}{%
1-\lambda }\leq \frac{(n-1)\lambda ^{2}}{1-\lambda },$ i.e.,
\begin{equation*}
\frac{1}{1-\lambda }\left( \lambda (2n-1)-1\right) \geq 0
\end{equation*}%
due to $\lambda \geq \frac{1}{2n-1}.$

Now, suppose
\begin{equation*}
2\beta \lambda ^{2}-\lambda =\sum\nolimits_{t=3}^{\infty }(2d_{t})\lambda
^{t}
\end{equation*}%
with $d_{t}\in \{-(n-1),\cdots ,-1,0,1,\cdots ,(n-1)\}.$

(1) When $2d_{t}\geq 0$ for $t\geq 3,$ let
\begin{equation*}
c_{i}^{(t)}=\left\{
\begin{array}{cc}
d_{t} & \text{if }i\text{ is odd,} \\
0 & \text{if }i\text{ is even.}%
\end{array}%
\right.
\end{equation*}

(2) When $2d_{t}<0$ for $t\geq 3,$ let
\begin{equation*}
c_{i}^{(t)}=\left\{
\begin{array}{cc}
0 & \text{if }i\text{ is odd,} \\
-d_{t} & \text{if }i\text{ is even.}%
\end{array}%
\right.
\end{equation*}%
Then by (\ref{1})-(\ref{2}), for
\begin{equation*}
c_{i}=(\lambda +\alpha )\sum\nolimits_{t=1}^{\infty }c_{i}^{(t)}\lambda ^{t}%
\text{ with }c_{2i-1}^{(1)}=i-1,c_{2i-1}^{(2)}=0\text{ and }%
c_{2i}^{(1)}=i-1,c_{2i}^{(2)}=\beta ,
\end{equation*}
equations (\ref{0}) hold. The step $(2)\Longrightarrow (3)$ is finished$.$
\section{One problem}
We pose the following question. 
\begin{Question}
Whether 
 $F(\lambda)=L_{AP}(E_{\lambda }^{2})$ is an increasing  staircase function with respect to $\lambda$. 
\end{Question}

\section*{Acknowledgements}
The work is supported by National Natural Science Foundation of China (Nos.
11831007, 11771226, 11701302, 11371329, 11471124, 11671147). The work is
also supported by K.C. Wong Magna Fund in Ningbo University.


\end{document}